\documentclass[12pt]{article}

\usepackage{amssymb}
\usepackage{amsthm}
\usepackage{amsmath}
\usepackage{graphicx}
\usepackage{fullpage}
\usepackage{subcaption}
\usepackage{color}
\usepackage{enumerate}
 \numberwithin{equation}{section}
\usepackage{booktabs}
\allowdisplaybreaks
\newtheorem{innercustomthm}{Theorem}
\newenvironment{customthm}[1]
  {\renewcommand\theinnercustomthm{#1}\innercustomthm}
  {\endinnercustomthm}
\usepackage[pdftex]{hyperref}
 \usepackage{mathtools}
 \mathtoolsset{showonlyrefs}
 \usepackage[nocompress]{cite}
\usepackage{url}
\theoremstyle{plain}
\newtheorem{thm}{Theorem}[section]

\newtheorem{lem}[thm]{Lemma}
\newtheorem{prop}[thm]{Proposition}

\newtheorem*{customlem}{Lemma}

\theoremstyle{definition}

\theoremstyle{remark}

\newtheorem{rem}[thm]{Remark}

\newcommand{\R}{\mathbb{R}}

\newcommand{\bp}{\begin{proof}[\ensuremath{\mathbf{Proof}}]}
\newcommand{\bs}{\begin{proof}[\ensuremath{\mathbf{Solution}}]}
\newcommand{\ep}{\end{proof}}
\newcommand{\be}{\begin{equation}}
\newcommand{\ee}{\end{equation}}

\begin{document}

\title{A volume formula for Reuleaux polyhedra}

\author{Ryan Hynd\footnote{Department of Mathematics, University of Pennsylvania. This work was supported in part by NSF grant DMS-2350454.}}

\maketitle

\begin{abstract}
A ball polyhedron is a finite intersection of congruent balls in $\mathbb{R}^3$. These shapes arise in various contexts in discrete and convex geometry.  We focus on Reuleaux polyhedra, the subclass of ball polyhedra whose centers and vertices coincide. Building on Bogosel’s recent work on the volume of Meissner polyhedra, we derive a formula for the volume of Reuleaux polyhedra in terms of their edges.
\end{abstract}

\section{Introduction}
A ball polyhedron is a finite intersection of closed balls in $\R^3$ with the same radius.  These shapes have commonalities and differences with standard convex polyhedra in $\R^3$ as discussed in \cite{MR2333996, MR2593321}. In addition, they have been used to study problems in discrete and convex geometry.  For example, ball polyhedra have been used to find the maximum number of diametric pairs a finite subset of $\R^3$ can have  \cite{Erdos1946, MR87115, MR87116, MR0087117} and to partition a given finite subset of $\R^3$ into four subsets with smaller diameter \cite{AgarwalPach1995,LopezCamposOliverosRamirezAlfonsin2025,HeppesRevesz1956}. They have also been employed to construct constant width polyhedra in $\R^3$ \cite{Meissner, MR4775724,MR3620844}.

\par This note concerns the volume of ball polyhedra whose centers are precisely the vertices appearing on their boundaries. Such bodies are called Reuleaux polyhedra. Our volume formula is expressed in terms of their circular boundary edges. We will restrict attention to unit balls, since the general case follows by scaling.

\par In earlier work \cite{MR4883943}, we derived a formula for the volume of a Reuleaux polyhedron from the coordinates of its vertices, after choosing an oriented cyclic ordering of the vertices on each face. Thus the volume was already known to be computable, but that formula did not show that it is determined by the edges of the shape. The purpose of the present note is to give such a formula, complementing Bogosel’s formulas for the surface area of a Reuleaux polyhedron and its associated Meissner polyhedra \cite{Bogosel2024Meissner}.  
\\
\par {\bf Reuleaux polyhedra.} Let $X$ be a finite subset of $\R^3$ with diameter one. It is a celebrated theorem of Gr\"unbaum \cite{MR87115}, Heppes \cite{MR87116}, 
and Straszewicz \cite{MR0087117} that if $X$ has $|X|\ge 4$ points, then there are at most $2|X|-2$ pairs $\{x,y\}\subset X$ with
$|x-y|=1$.  When $X$ has $2|X|-2$ diametric pairs, we will say that $X$ is {\it extremal}. The simplest example of an extremal subset is the vertices 
of a regular tetrahedron whose edge lengths are all equal to one.  However, there 
are many other examples as discussed in the references \cite{MR2593321, MR3930585, MR3620844}. 

\par For a given extremal $X\subset \R^3$, we will consider its associated {\it Reuleaux polyhedron}
$$
B(X)=\bigcap_{x\in X}B(x). 
$$
Here and throughout $B(x)\subset \R^3$ will denote the closed unit ball centered at $x$; we will also use the  notation $B(X)$ for the intersection above for any given and not necessarily finite $X\subset \R^3$.  The faces of $B(X)$ are the subsets 
$$
\partial B(x)\cap B(X), \quad x\in X
$$
of its boundary. There is one face for each $x\in X$, and each face is a geodesically convex subset of its respective sphere.  

\par The {\it principal vertices} of $B(X)$ are boundary points  which belong to at least three faces.  The {\it dangling vertices} are points $x\in X$ which belong to exactly two faces. The {\it vertex set} of $B(X)$ is the union of principal and dangling vertices. A fundamental result due to
Kupitz, Martini, and Perles is that the vertex set of $B(X)$ coincides with the set of centers $X$ \cite{MR2593321}. 
\par The {\it edges} of $B(X)$ are connected components of 
$$
\partial B(x)\cap \partial B(y)\cap B(X)\setminus X
$$
as $x,y$ range over $X$.  It can be shown that each edge is a circular arc of radius less than one and the endpoints of each edge are vertices.  
Moreover, if 
$e\subset \partial B(x)\cap \partial B(y)$ is an edge with endpoints $x',y'\in X$, then there is a unique edge 
$e'\subset \partial B(x')\cap \partial B(y')$ with endpoints $x,y\in X$ (Theorem 8.1 of \cite{MR2593321}). In this case, we will say that $(e,e')$ is a   {\it dual edge pair}. 
Finally, a simple application of Euler's formula (Proposition 6.2 of \cite{MR2593321}) for ball polyhedron gives that $B(X)$ has $|X|-1$ dual edge pairs. 
\\
 \begin{figure}[h]
\centering
 \includegraphics[width=.36\textwidth]{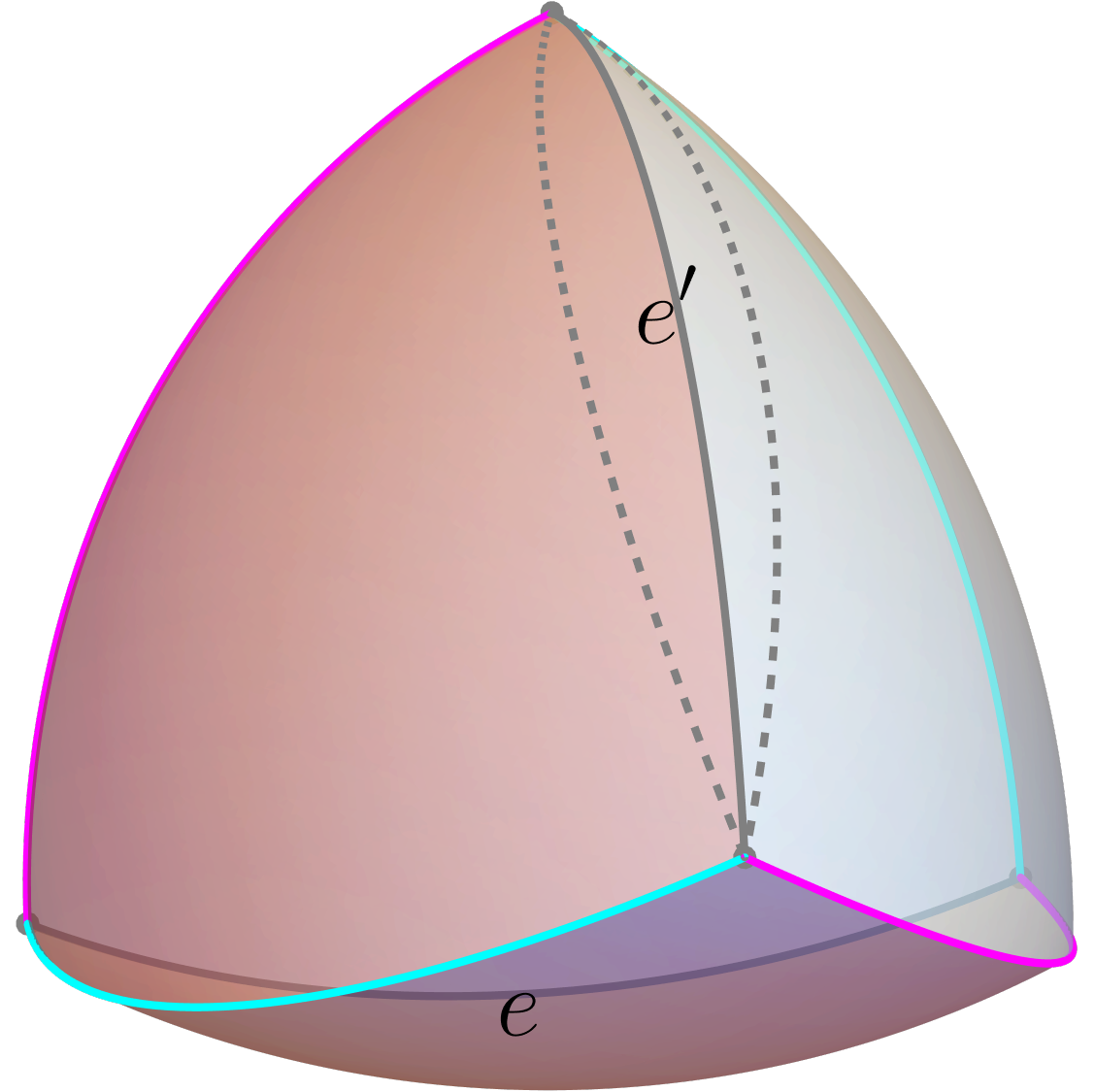}
 \hspace{.4in}
  \includegraphics[width=.36\textwidth]{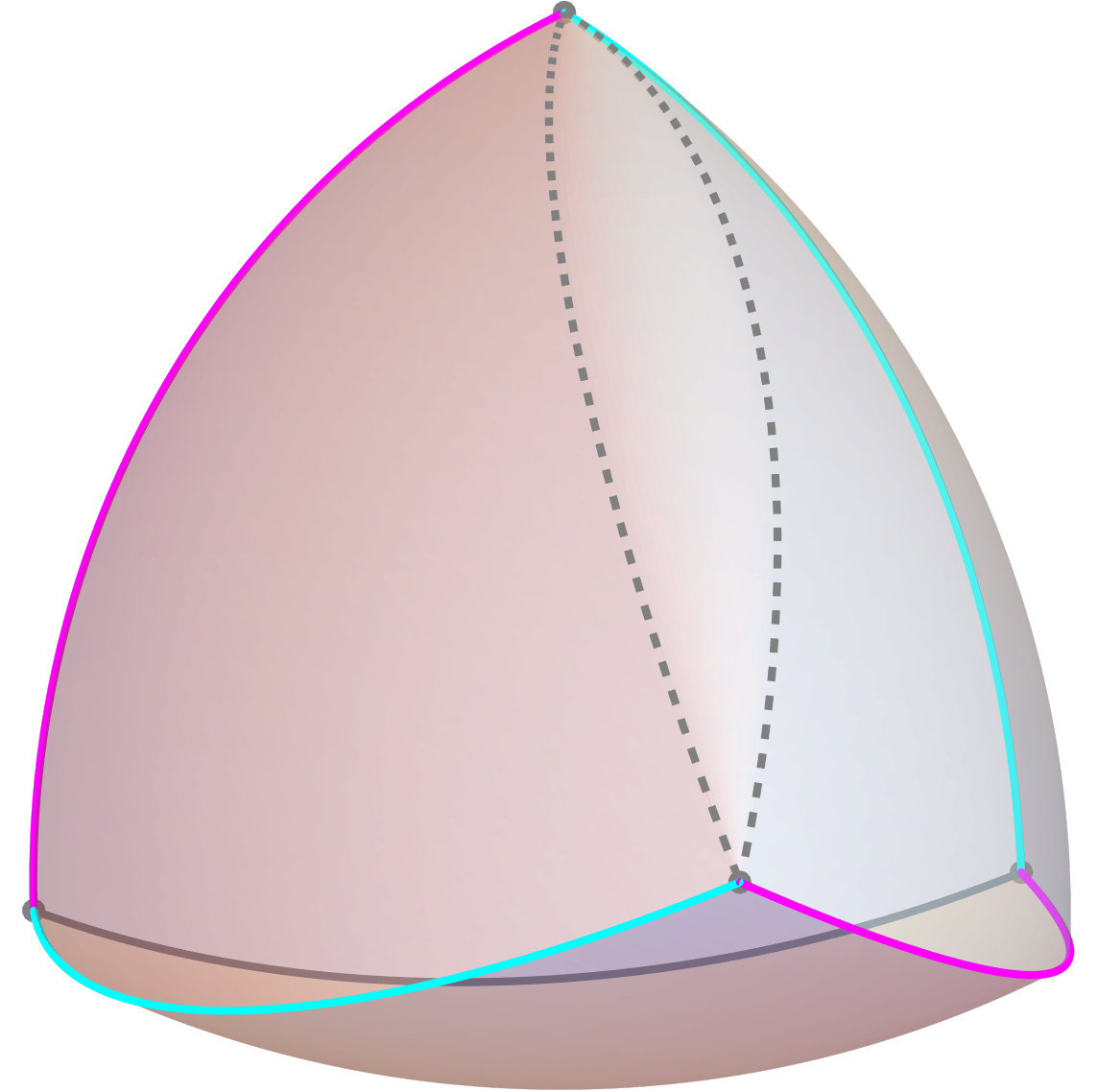}
 \caption{On the left is a Reuleaux polyhedron $B(X)$, where $X$ is the vertex set of a regular tetrahedron with side length one. Note that there are three dual edge pairs, and the two edges in each pair are shown in bold and in the same color. For example, the pair $(e,e')$ is labeled in gray. Also note that the endpoints of the edge $e'$ are joined by two dashed geodesics in their respective spheres. On the right is $B(X\cup e)$, which is the figure obtained by performing surgery on $B(X)$ near $e'$. In particular, this is the figure whose boundary is obtained by cutting out the region of $\partial B(X)$ bounded by the two geodesics (which contains $e'$) and rotating one geodesic into the other about the line passing through the endpoints of $e'$. }\label{ReulFig}
\end{figure}
\par {\bf Meissner polyhedra.}  Let us suppose that $X$ is extremal and that $(e,e')$ is a dual edge pair of $B(X)$. In addition, 
let us denote the endpoints of $e$ by $b,c\in X$ and the endpoints of $e'$ as $b',c'\in X$. The geodesic $\gamma_b\subset \partial B(b)$ 
which joins $b'$ to $c'$ is also included in $\partial B(b)\cap B(X)$, as this face is geodesically convex in $\partial B(b)$.  Likewise, 
the geodesic $\gamma_c\subset \partial B(c)$ 
which joins $b'$ to $c'$ is included in $\partial B(c)\cap B(X)$. If we remove from $\partial B(X)$ the regions lying between $e'$ and $\gamma_b$ and between $e'$ and $\gamma_c$, and replace them by the surface obtained by rotating $\gamma_b$ into $\gamma_c$ about the line through $b'$ and $c'$, then the resulting surface bounds the convex body $B(X\cup e)$.  See Figure \ref{ReulFig}.  We interpret $B(X\cup e)$ as the body obtained by performing surgery on $B(X)$ near $e'$.

\par Set $m=|X|$. The convex body 
\be\label{MeissPolyExpr}
M=B(X\cup e_1\cup\dots \cup e_{m-1})
\ee
is known as a {\it Meissner polyhedron}. As discussed above, $M$ is obtained from $B(X)$ by performing surgery on $B(X)$ near the edges $e_1',\dots, e_{m-1}'$; see Figure \ref{MeissFig} for an example.  Meissner and Schilling \cite{Meissner} proved that $M$ has constant width one when $X$ is the vertex set of a regular tetrahedron, and Montejano and Rold\'an-Pensado \cite{MR3620844} established this result for general $X$. That is, the distance between any pair of parallel supporting planes to $M$ is equal to one.  
 \begin{figure}[h]
\centering
 \includegraphics[width=.36\textwidth]{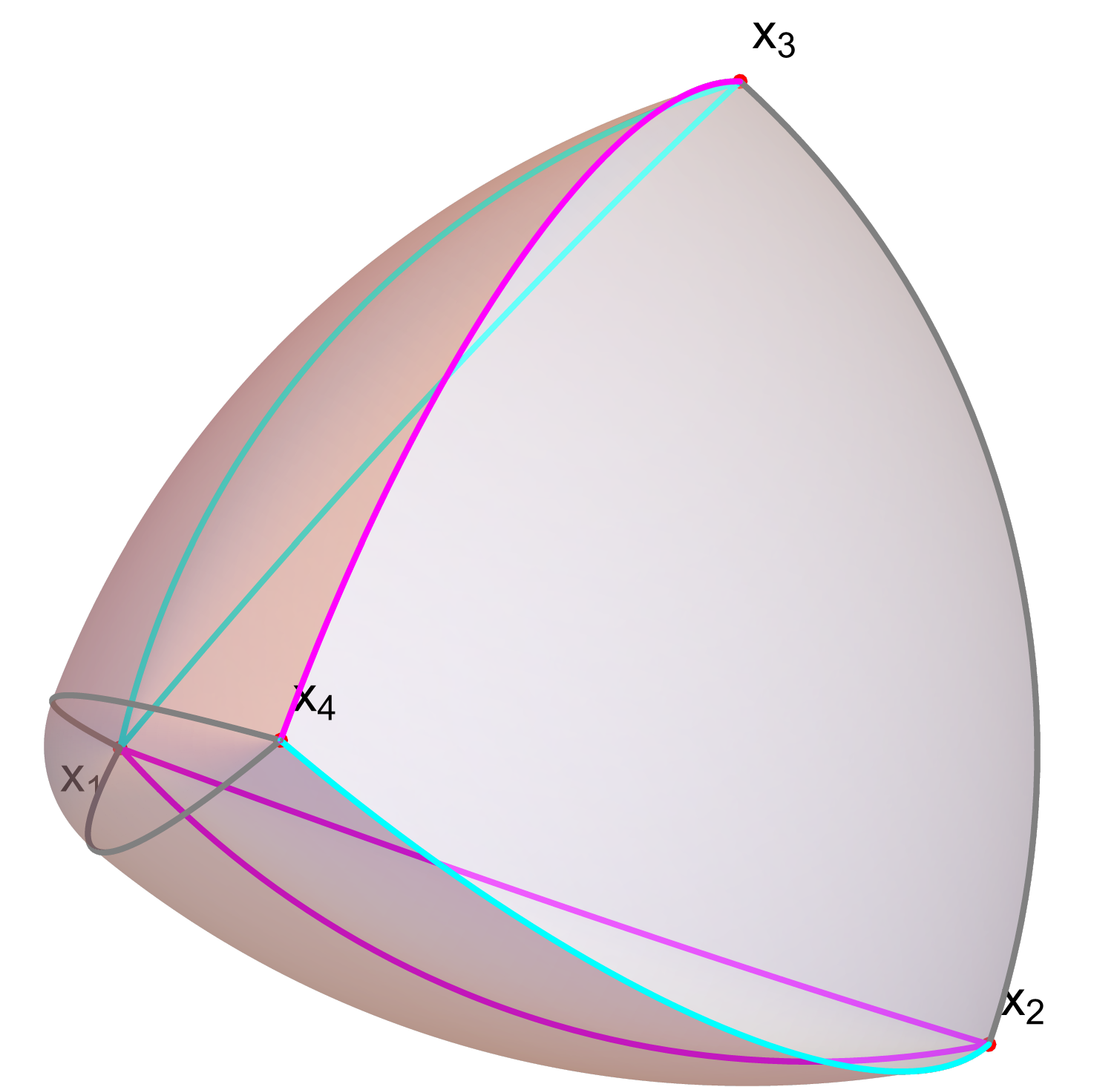}
 \hspace{.4in}
  \includegraphics[width=.36\textwidth]{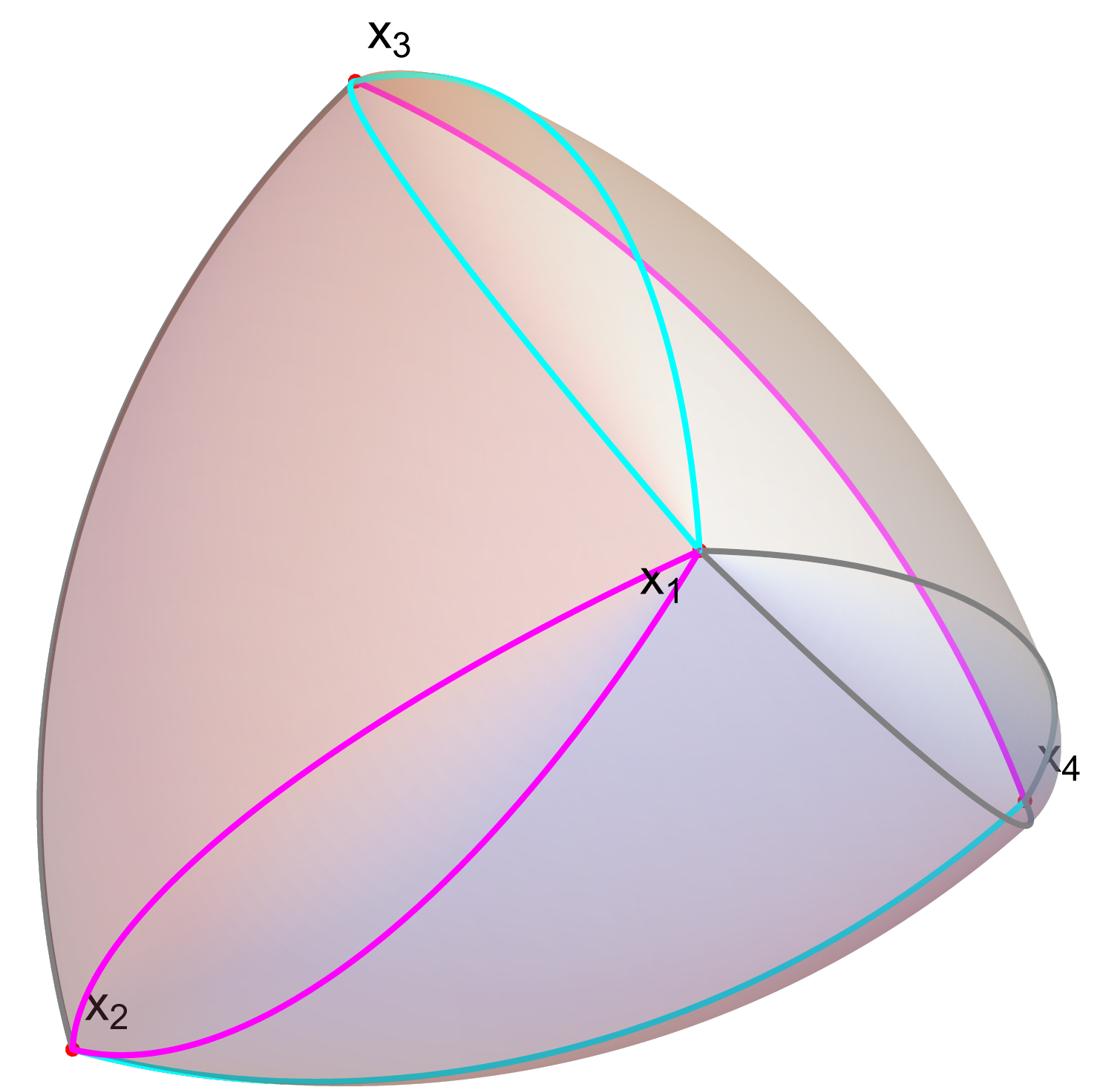}
 \caption{These are two views of a Meissner polyhedron $M=B(X\cup e_1\cup e_2\cup e_3)$, where $X$ is the vertex set of a regular tetrahedron. The dual edge pairs are $(e_1,e_1'), (e_2,e_2')$, and $(e_3,e_3')$, and $M$ is obtained from $B(X)$ by performing surgery near the edges $e_1', e_2'$, and $e_3'$.  This is a constant width body in $\R^3$, and its volume can be computed using formula \eqref{BogoselVolForm}. }\label{MeissFig}
\end{figure}
\par {\bf Bogosel's formulas.}  A basic problem is to compute the volume of the Meissner polyhedron $M$ specified in \eqref{MeissPolyExpr}. This was accomplished recently by Bogosel \cite{Bogosel2024Meissner}, who first showed that the surface area of $M$ is given by 
$$
S(M)=2\pi-\sum^{m-1}_{i=1}f(\theta(e_i),\theta(e_i')).
$$
Here
$$
f(\theta,\theta')=2 \sin^{-1}\left(\frac{\sin(\theta/2)}{\cos(\theta'/2)}\right)\theta'\cos(\theta'/2),
$$
and $\theta(e),\theta(e')\in (0,\pi/3]$ are the lengths of geodesics which join the endpoints of the dual edges $e,e'$ in either of the spheres they respectively border.  See \eqref{thetaDefn} and \eqref{thetaPrimeDefn} below for precise definitions. 

\par Next, Bogosel applied Blaschke's relation 
\be\label{BlaschkeRelation}
V(M)=\frac{1}{2}S(M)-\frac{\pi}{3}
\ee
to conclude the volume formula
\be\label{BogoselVolForm}
V(M)=\frac{2\pi}{3}-\frac{1}{2}\sum^{m-1}_{i=1}f(\theta(e_i),\theta(e_i')).
\ee
He also mentioned in passing that the surface area of $B(X)$ is
$$
S(B(X))=2\pi-\sum^{m-1}_{i=1}g(\theta(e_i),\theta(e_i')),
$$
where 
\begin{align}
g(\theta,\theta')&=4\Bigg\{\sin^{-1}\left(\frac{\sin(\theta/2)}{\cos(\theta'/2)}\right)\sin(\theta'/2)   +
\sin^{-1}\left(\frac{\sin(\theta'/2)}{\cos(\theta/2)}\right)\sin(\theta/2) \\
&\hspace{1in}-\sin^{-1}\Big(\tan(\theta/2)\tan(\theta'/2)\Big)\Bigg\}.
\end{align}
\par {\bf Main result}. In what follows, we will derive a formula for the volume of $B(X)$ analogous to the formulas above. 
This formula involves the function
\begin{align}\label{hFunction}
h(\theta,\theta')&=4\Bigg\{ \sin^{-1}\left(\frac{\sin(\theta'/2)}{\cos(\theta/2)}\right)\left(\sin(\theta/2)-\frac{1}{3}\sin(\theta/2)^3\right)   \\
&\hspace{.3in} +\sin^{-1}\left(\frac{\sin(\theta/2)}{\cos(\theta'/2)}\right)\left(\sin(\theta'/2)-\frac{1}{3}\sin(\theta'/2)^3\right) \\
&\hspace{.3in}  -\frac{2}{3}\sin^{-1}\Big(\tan(\theta/2)\tan(\theta'/2)\Big) -\frac{1}{3}\sin(\theta'/2)\sin(\theta/2)\sqrt{1-\sin(\theta/2)^2-\sin(\theta'/2)^2}\Bigg\}.
\end{align}
 \begin{customthm}{(Volume Formula)}\label{mainthm}
The volume of $B(X)$ is given by
\be\label{VolumeFormula}
V(B(X))=\frac{2\pi}{3}-\frac{1}{2}\sum^{m-1}_{i=1}h(\theta(e_i),\theta(e_i')).
\ee
\end{customthm}

\par For example, suppose $Z$ is the collection of vertices of a regular tetrahedron with side lengths equal to one. As the faces of this tetrahedron 
are equilateral triangles, $\theta(e)=\theta(e')=\pi/3$ for each dual edge pair $(e,e')$ of $B(Z)$.  Therefore, 
the volume of $B(Z)$ is equal to $2\pi/3-(3/2)h(\pi/3,\pi/3)$. Using standard trigonometric identities, we 
additionally find 
$$
V(B(Z))=\frac{8 \pi}{3}+\frac{\sqrt{2}}{4}-\frac{27}{4}  \cos^{-1}\left(\frac{1}{3}\right).
$$
This formula was first derived using the symmetry of the figure by Harbourne \cite{harbourne}. 

\par We will also explain below that 
$h(\theta,\theta')>g(\theta,\theta')$ for $\theta,\theta'\in (0,\pi/3]$. Combining this observation with Bogosel's 
surface area formula implies 
\be\label{AlmostBlaschkeIneq}
V(B(X))<\frac{1}{2}S(B(X))-\frac{\pi}{3}
\ee
for all Reuleaux polyhedra. Recall that equality would hold in the above inequality if $B(X)$ is replaced by $M$ as stated in \eqref{BlaschkeRelation}. 
\\ 
\par This paper is organized as follows.  We will discuss our approach to proving \eqref{VolumeFormula} in Section \ref{PrelimSect}.  
This involves the evaluation of various integrals which we will do in Sections \ref{SliverSect} and \ref{SpindleSect}. In Section \ref{ProofSect}, we will summarize our findings and verify the volume formula. Finally, we will prove \eqref{AlmostBlaschkeIneq} in the appendix.

\section{Approach}\label{PrelimSect}
We will now describe a basic approach to establishing our volume formula \eqref{VolumeFormula}. 
To this end, we will fix an extremal $X\subset \R^3$ with $m\ge 4$ points and label the dual edge pairs  of 
the corresponding Reuleaux polyhedron $B(X)$ as $(e_1,e_1'),\dots, (e_{m-1},e_{m-1}')$.  Again 
we define the Meissner polyhedron $M$ by \eqref{MeissPolyExpr}. Recall that this shape is obtained by performing surgery on $B(X)$ near the edges $e_1',\dots, e_{m-1}'$. 

\par Observe 
$$
B(X)=M\cup W(e_1')\cup \dots \cup W(e_{m-1}'),
$$ 
where $W(e_1'),\dots, W(e_{m-1}')$ are the {\it wedge regions} which have been cut away 
from $B(X)$ during the surgery process done to construct $M$. Moreover, the interiors of $M, W(e_1'), \dots , W(e_{m-1}')$ are 
disjoint; see Section 4 of \cite{MR4775724}. It follows that 
\begin{align}
V(B(X))&=V(M)+\sum^{m-1}_{i=1}V(W(e_i'))\\
&=\frac{2\pi}{3}-\frac{1}{2}\sum^{m-1}_{i=1}f(\theta(e_i),\theta(e_i'))+\sum^{m-1}_{i=1}V(W(e_i'))\\
&=\frac{2\pi}{3}+\sum^{m-1}_{i=1}\left[V(W(e_i'))- \frac{1}{2}f(\theta(e_i),\theta(e_i'))\right].
\end{align}
Therefore, in order to verify \eqref{VolumeFormula}, we only need to prove the following lemma.  
\begin{customlem}[Wedge Lemma]
For any dual edge pair $(e,e')$ of $B(X)$, 
\be\label{CrucialhIdentity}
V(W(e')) = \frac{1}{2}f(\theta(e),\theta(e'))-\frac{1}{2}h(\theta(e),\theta(e')).
\ee
\end{customlem} 
\noindent  This will be a standard computation once we specify wedge regions and other data involved in this problem. 
\\
\par {\bf Wedge region.} Let us fix a dual edge pair $(e,e')$ of $B(X)$. We will denote the endpoints of $e$ by $b, c\in X$ and the endpoints of 
$e'$ by  $b', c'\in X$. We may also suppose  
\be\label{OrientedPoints}
\left(b'-\frac{b+c}{2} \right)\times \left(c'-\frac{b+c}{2} \right)\cdot (b-c)>0
\ee
as indicated in Figure \ref{WedgeFig}. Consider the geodesic $\gamma_b\subset \partial B(b)$ which joins $b'$ and $c'$ and the geodesic $\gamma_c\subset \partial B(c)$ which joins $b'$ and $c'$. Denote $R_b\subset \partial B(b)$ as the surface bounded by $\gamma_b$ and $e'$, $R_c\subset\partial B(c)$ as the surface bounded by $\gamma_c$ and $e'$, and $R_s$ as the surface obtained by rotating $\gamma_b$ into $\gamma_c$ counterclockwise about the line which passes through $b'$ and $c'$ with direction vector $b'-c'$. Also note  $W(e')$ is the region bounded by $R_b$, $R_c$ and $R_s$.  

\begin{figure}[h]
\centering
 \includegraphics[width=.7\textwidth]{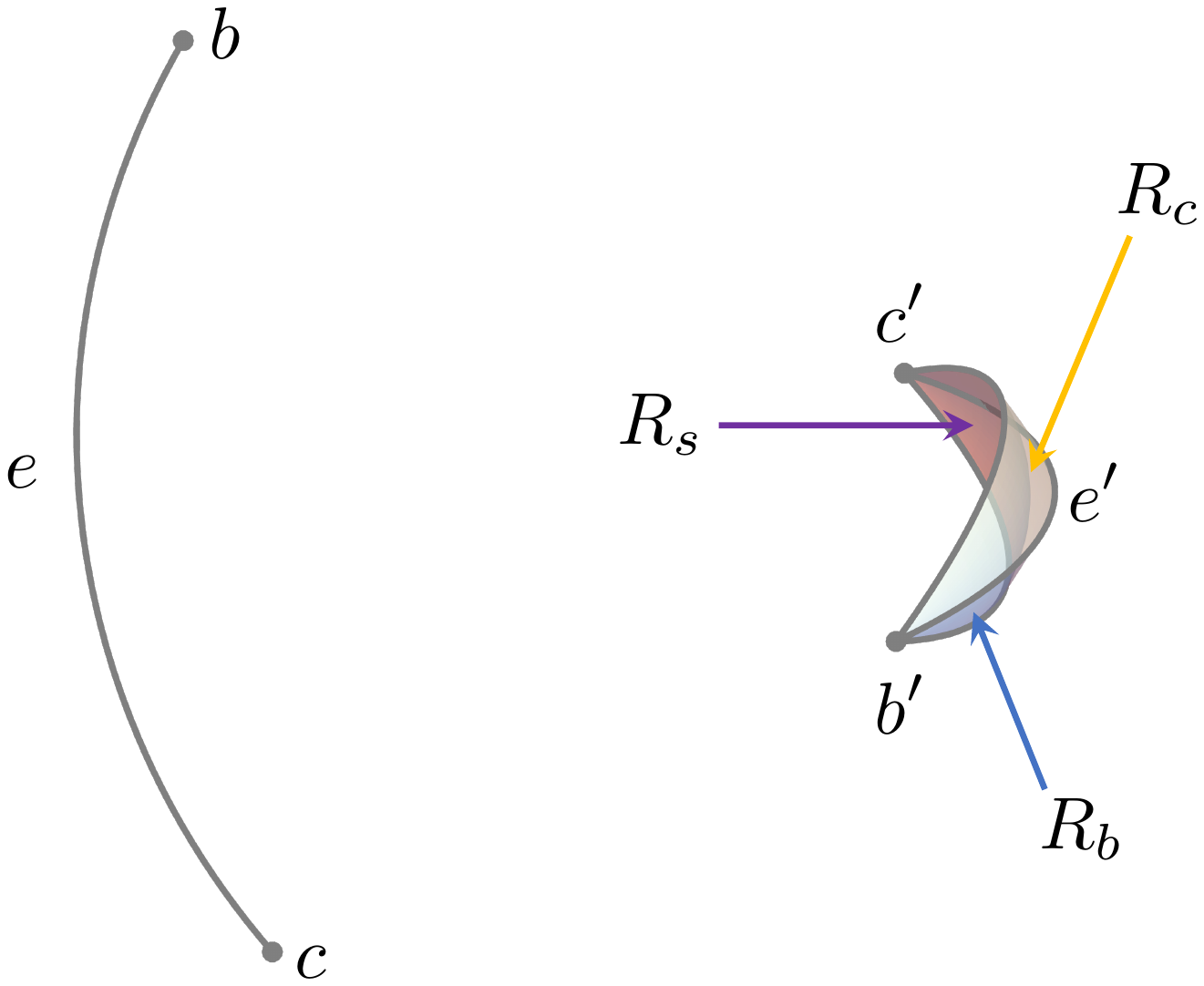}
 \caption{This is a diagram of a dual edge pair $(e,e')$ and associated wedge region $W(e')$ which is bounded by the three surfaces $R_b, R_c$ and $R_s$. }
 \label{WedgeFig}
\end{figure}
\par Employing \eqref{OrientedPoints}, it is not hard to check that the surface $R_c$ is the set of $x\in \R^3$ with 
$$
|x-c|=1,\quad \left(x-\frac{b+c}{2}\right)\cdot (b-c)\ge 0, \quad (x-c)\cdot (b'-c)\times (c'-c)\le 0.
$$
Likewise $x\in R_b$ if and only if 
$$
|x-b|=1,\quad \left(x-\frac{b+c}{2}\right)\cdot (b-c)\le 0, \quad (x-b)\cdot (b'-b)\times (c'-b)\ge 0.
$$
It is also evident that $R_b$ is the reflection of $R_c$ about the plane with normal $b-c$ which passes through $(b+c)/2$. 

\par The surface obtained by rotating $\gamma_b$ about the line which passes 
through $b'$ and $c'$ is given by the set of points $x\in \R^3$ that satisfy the equation
\begin{align}\label{spindleSurface}
\left|x-\frac{b'+c'}{2}-\left(\left(x-\frac{b'+c'}{2}\right)\cdot \frac{b'-c'}{|b'-c'|}\right) \frac{b'-c'}{|b'-c'|}\right|
+\sqrt{1-\left|\frac{b'-c'}{2}\right|^2}\\
\hspace{1in}=\sqrt{1-\left(\left(x-\frac{b'+c'}{2}\right)\cdot \frac{b'-c'}{|b'-c'|}\right)^2}
\end{align}
(see Section 2 of \cite{MR4775724}). Therefore, $R_s$ is the portion of this surface 
which lies in the intersection of the half-spaces  
\begin{align}\label{spindleCondition}
(x-c)\cdot (b'-c)\times (c'-c)\le 0\quad \text{and}\quad (x-b)\cdot (b'-b)\times (c'-b)\ge 0.
\end{align}

\par {\bf Volume.} Our goal is to compute the volume of $W(e')$. Let us write $n$ for the outward unit normal on the piecewise smooth, oriented surface $\partial W(e')$. By the divergence theorem, 
\begin{align}
V(W(e'))&=\frac{1}{3}\int_{ W(e')}\nabla\cdot xdV\\
&=\frac{1}{3}\int_{\partial W(e')}x\cdot ndS\\
&=\frac{1}{3}\int_{R_b}x\cdot ndS+\frac{1}{3}\int_{R_c}x\cdot ndS+\frac{1}{3}\int_{R_s}x\cdot ndS.
\end{align}
We are left to compute each of the three ``volume" integrals above. We will do so in terms of the angles of interest associated with the dual edge pair $(e,e')$. 
\\
\par {\bf Angles.}  The geodesic distance between the endpoints of $e$ in $\partial B(b')$ or $\partial B(c')$ is
\be\label{thetaDefn}
\theta(e)=2\sin^{-1}\left(\frac{|b-c|}{2}\right).
\ee
It follows that $\theta(e)$ is the angle between $b-b'$ and $c-b'$ and is also the angle between $b-c'$ and $c-c'$. Likewise
\be\label{thetaPrimeDefn}
\theta(e')=2\sin^{-1}\left(\frac{|b'-c'|}{2}\right)
\ee
is the angle between $b'-b$ and $c'-b$ and is also the angle between $b'-c$ and $c'-c$. 

\par As noted by Bogosel \cite{Bogosel2024Meissner}, the dihedral angle $\phi(e)$ between the plane which contains $b, c$ and $b'$ and 
the plane which contains $b, c$ and $c'$ fulfills 
$$
\sin(\phi(e)/2)=\frac{\sin(\theta(e')/2)}{\cos(\theta(e)/2)}.
$$
Similarly, the  dihedral angle $\phi(e')$ between the plane which contains $b', c'$ and $b$ and 
the plane which contains $b', c'$ and $c$ satisfies 
  $$
\sin(\phi(e')/2)=\frac{\sin(\theta(e)/2)}{\cos(\theta(e')/2)}.
$$
A crucial identity between these angles is 
\be\label{MidPointIdentity}
\cos(\theta(e)/2)\cos(\phi(e)/2)=\cos(\theta(e')/2)\cos(\phi(e')/2).
\ee
This quantity is also equal to the distance between the midpoints of the line segments joining the endpoints of $e$ and $e'$. 
\\
\par {\bf Special coordinates.}  To simplify our computation, we will write 
$$
\theta=\theta(e), \quad \theta'=\theta(e'), \quad 
\phi=\phi(e), \quad \phi'=\phi(e').
$$
Since volume is invariant under orthogonal transformations and translations, we will also assume that the endpoints of $e,e'$ satisfy 
\be\label{SpecialCoordinates}
\begin{cases}
b=ae_3\\
c=-ae_3\\
b'=\sqrt{1-a^2}e_1\\
c'=\sqrt{1-a^2}(\cos\phi, \sin\phi,0).
\end{cases}
\ee
Here we have written $a=|b-c|/2$ and $\{e_1,e_2,e_3\}$ for the standard basis in $\R^3$.\footnote{We hope the context will make it clear whether or not $e_i$ is an edge or a standard unit vector in $\R^3$.} And with this choice of coordinates,  
$$
\theta=2\sin^{-1}(a)\quad \text{and}\quad \theta'=2\sin^{-1}\left(\sqrt{1-a^2}\sin(\phi/2)\right).
$$

\section{Sliver integrals}\label{SliverSect}
We will first evaluate the integral 
$$
\int_{R_c}x\cdot ndS
$$
Note that in the coordinates \eqref{SpecialCoordinates} $R_c$ is given by
$$
|x+ae_3|=1, \quad x\cdot e_3\ge 0,\quad (x+ae_3)\cdot u\le 0.
$$
Here  
\begin{align}
u&=(b'-c)\times (c'-c)\\
&=(\sqrt{1-a^2},0,a)\times (\sqrt{1-a^2}\cos\phi,\sqrt{1-a^2}\sin\phi,a)\\
&=2\sqrt{1-a^2}\sin(\phi/2) (-a\cos(\phi/2),-a\sin(\phi/2),\sqrt{1-a^2}\cos(\phi/2)).
\end{align}

\par We may parametrize $\partial B(-ae_3)$ with spherical coordinates 
$$
Y(\psi,t)=-ae_3+\left(\cos(t)\cos(\psi) , \cos(t)\sin(\psi) ,\sin(t) \right) 
$$
for $\psi\in [0,2\pi]$ and $t\in [-\pi/2,\pi/2]$. In order to parametrize the sliver, we need to specify the $(\psi,t)$ for which
$$
Y(\psi,t)\cdot e_3\ge 0\quad\text{and}\quad (Y(\psi,t)+ae_3)\cdot u\le 0.
$$ 
 The inequality  $Y(\psi,t)\cdot e_3\ge 0$ is equivalent to 
$$ t\ge \sin^{-1}(a).$$  And it is routine to check that $(Y(\psi,t)+ae_3)\cdot u\le 0$ is the same as requiring
$$
t\le \sin^{-1}\left( \frac{a\cos(\psi-\phi/2)}{\sqrt{(1-a^2)\cos(\phi/2)^2+(a\cos(\psi-\phi/2))^2}}\right)
$$
and $0\le \psi\le \phi$.

\subsection{Area integral}
Observe that the 
Jacobian associated with the parametrization $Y(\psi,t)$ is 
$$
|\partial_\psi Y\times \partial_t Y|=\cos(t).
$$
It follows that 
\begin{align}
S(R_c)&=\int^\phi_0\int^{\sin^{-1}\left( \frac{a\cos(\psi-\phi/2)}{\sqrt{(1-a^2)\cos(\phi/2)^2+(a\cos(\psi-\phi/2))^2}}\right)}_{\sin^{-1}a}\cos(t)dtd\psi\\
&=\int^\phi_0 \frac{a\cos(\psi-\phi/2)}{\sqrt{(1-a^2)\cos(\phi/2)^2+(a\cos(\psi-\phi/2))^2}}d\psi-\phi a\\
&=2\int^{\phi/2}_{0} \frac{a\cos(s)}{\sqrt{(1-a^2)\cos(\phi/2)^2+(a\cos(s))^2}}ds-\phi a\\
&=2\left.\sin^{-1}\left(\frac{a\sin(s)}{\sqrt{(1-a^2)\cos(\phi/2)^2+a^2}}\right)\right|^{\phi/2}_0-\phi a\\
&=2\sin^{-1}\left(\frac{a\sin(\phi/2)}{\sqrt{(1-a^2)\cos(\phi/2)^2+a^2}}\right)-\phi a.
\end{align} 

\par Also recall $a=\sin(\theta/2)$ and note 
\begin{align}
(1-a^2)\cos(\phi/2)^2+a^2&=1-(1-a^2)\sin(\phi/2)^2=1-\sin(\theta'/2)^2=\cos(\theta'/2)^2. 
\end{align}
As a result, 
 \begin{align}
 \frac{a\sin(\phi/2)}{\sqrt{(1-a^2)\cos(\phi/2)^2+a^2}}=\frac{\sin(\theta/2)\displaystyle\frac{\sin(\theta'/2)}{\cos(\theta/2)}}{\cos(\theta'/2)}
 =\tan(\theta/2)\tan(\theta'/2).
 \end{align}
 We conclude 
 \begin{align}
 S(R_c)&=2\sin^{-1}\Big(\tan(\theta/2)\tan(\theta'/2)\Big)-\sin(\theta/2)\phi.
 \end{align}
\begin{rem}
This surface area can be derived using the Gauss--Bonnet formula. See Proposition 3.4 of \cite{Bogosel2024Meissner}
and Lemma 8.1 of \cite{MR4883943}. 
\end{rem}
 
\subsection{Volume integral}
Observe that $n=x+ae_3$ on $R_c$, so that  
$$
x\cdot n=(n-ae_3)\cdot n=1-an_3.
$$
As the outer unit normal corresponding to parametrization $Y(\psi,t)$ is 
$$
n\circ Y(\psi,t)=\left(\cos(t)\cos(\psi) , \cos(t)\sin(\psi) ,\sin(t) \right),
$$
we have
$$
Y(\psi,t)\cdot n\circ Y(\psi,t)=1-a\sin(t).
$$
It follows that  
\begin{align}
\int_{R_c}x\cdot ndS&=\int^\phi_0\int^{\sin^{-1}\left( \frac{a\cos(\psi-\phi/2)}{\sqrt{(1-a^2)\cos(\phi/2)^2+(a\cos(\psi-\phi/2))^2}}\right)}_{\sin^{-1}a}(1-a\sin(t))\cos(t)dtd\psi\\
&=S(R_c)-a\int^\phi_0\int^{\sin^{-1}\left( \frac{a\cos(\psi-\phi/2)}{\sqrt{(1-a^2)\cos(\phi/2)^2+(a\cos(\psi-\phi/2))^2}}\right)}_{\sin^{-1}a}\frac{d}{dt}\frac{1}{2}\sin(t)^2dtd\psi\\
&=S(R_c)-\frac{a}{2}\int^\phi_0\left[\frac{a^2\cos(\psi-\phi/2)^2}{(1-a^2)\cos(\phi/2)^2+(a\cos(\psi-\phi/2))^2} -a^2\right]d\psi\\
&=S(R_c)-a\int^{\phi/2}_{0}\left[\frac{a^2\cos(s)^2}{(1-a^2)\cos(\phi/2)^2+(a\cos(s))^2} -a^2\right]ds\\
&=S(R_c)-a\int^{\phi/2}_{0}\left[1 -a^2-\frac{(1-a^2)\cos(\phi/2)^2}{(1-a^2)\cos(\phi/2)^2+(a\cos(s))^2}\right]ds\\
&=S(R_c)-a(1-a^2)\frac{\phi}{2}+a\int^{\phi/2}_{0}\frac{(1-a^2)\cos(\phi/2)^2}{(1-a^2)\cos(\phi/2)^2+(a\cos(s))^2}ds.
\end{align} 

\par Direct computation shows 
\begin{align}
&\frac{d}{ds}\left[\frac{\sqrt{1-a^2}\cos(\phi/2)}{\sqrt{a^2+(1-a^2)\cos(\phi/2)^2}}\tan^{-1}\left(\tan(s)\frac{\sqrt{1-a^2}\cos(\phi/2)}{\sqrt{a^2+(1-a^2)\cos(\phi/2)^2}}\right)\right]\\
&\hspace{1in} = \frac{(1-a^2)\cos(\phi/2)^2}{(1-a^2)\cos(\phi/2)^2+(a\cos(s))^2}.
\end{align} 
Consequently,  
\begin{align}
&\int^{\phi/2}_{0}\frac{(1-a^2)\cos(\phi/2)^2}{(1-a^2)\cos(\phi/2)^2+(a\cos(s))^2}ds\\
&\quad = \frac{\sqrt{1-a^2}\cos(\phi/2)}{\sqrt{a^2+(1-a^2)\cos(\phi/2)^2}}\tan^{-1}\left(\tan(\phi/2)\frac{\sqrt{1-a^2}\cos(\phi/2)}{\sqrt{a^2+(1-a^2)\cos(\phi/2)^2}}\right)\\
&\quad = \frac{\sqrt{1-a^2}\cos(\phi/2)}{\sqrt{a^2+(1-a^2)\cos(\phi/2)^2}}\tan^{-1}\left(\frac{\sqrt{1-a^2}\sin(\phi/2)}{\sqrt{a^2+(1-a^2)\cos(\phi/2)^2}}\right)\\
&\quad = \frac{\cos(\theta/2)\cos(\phi/2)}{\cos(\theta'/2)}\tan^{-1}\left(\frac{\cos(\theta/2)\sin(\phi/2)}{\cos(\theta'/2)}\right)\\
&\quad = \frac{\cos(\theta/2)\cos(\phi/2)}{\cos(\theta'/2)}\tan^{-1}\left(\frac{\sin(\theta'/2)}{\cos(\theta'/2)}\right)\\
&\quad =\cos(\phi'/2)\; \theta'/2.
\end{align} 
Here we used \eqref{MidPointIdentity}. In summary, 
\begin{align}\label{IntComp1}
\int_{R_c}x\cdot ndS&=S(R_c)-a(1-a^2)\frac{\phi}{2}+a\cos(\phi'/2)\;\frac{ \theta'}{2}\nonumber \\
&=2\sin^{-1}\Big(\tan(\theta/2)\tan(\theta'/2)\Big)-\phi\sin(\theta/2) \\
&\hspace{1in}-\sin(\theta/2)(1-\sin(\theta/2)^2)\frac{\phi}{2}+\sin(\theta/2)\cos(\phi'/2)\; \frac{ \theta'}{2}\\
&=2\sin^{-1}\Big(\tan(\theta/2)\tan(\theta'/2)\Big)-\frac{3}{2}\phi\sin(\theta/2)+\sin(\theta/2)^3\frac{\phi}{2}+\sin(\theta/2)\cos(\phi'/2)\frac{\theta'}{2}.
\end{align} 

\subsection{Other volume integral}
In the coordinates \eqref{SpecialCoordinates}, the sliver region $R_b$ is given by
$$
|x-ae_3|=1, \quad x\cdot e_3\le 0,\quad (x-ae_3)\cdot v\ge 0.
$$
Here
\begin{align}
v&=(b'-b)\times (c'-b)\\
&=(\sqrt{1-a^2},0,-a)\times (\sqrt{1-a^2}\cos\phi,\sqrt{1-a^2}\sin\phi,-a)\\
&=2\sqrt{1-a^2}\sin(\phi/2) (a\cos(\phi/2),a\sin(\phi/2),\sqrt{1-a^2}\cos(\phi/2)).
\end{align}
As we argued above, this region can be parametrized with 
$$
Z(\psi,t)=ae_3+\left(\cos(t)\cos(\psi) , \cos(t)\sin(\psi) ,\sin(t) \right) 
$$
for 
$$
- \sin^{-1}\left( \frac{a\cos(\psi-\phi/2)}{\sqrt{(1-a^2)\cos(\phi/2)^2+(a\cos(\psi-\phi/2))^2}}\right)\le t\le -\sin^{-1}(a)
$$
and
$$
0\le\psi\le\phi.
$$

\par Since $n=x-ae_3$ for $x\in R_b$, 
$$
x\cdot n=(n+ae_3)\cdot n=1+an_3.
$$
It follows that
\begin{align}
\int_{R_b}x\cdot ndS&=\int_{R_b}(1+an_3)dS\\
&=\int^\phi_0\int^{-\sin^{-1}(a)}_{-\sin^{-1}\left( \frac{a\cos(\psi-\phi/2)}{\sqrt{(1-a^2)\cos(\phi/2)^2+(a\cos(\psi-\phi/2))^2}}\right)}(1+a\sin(t))\cos(t)dtd\psi\\
&=\int^\phi_0\int^{\sin^{-1}\left( \frac{a\cos(\psi-\phi/2)}{\sqrt{(1-a^2)\cos(\phi/2)^2+(a\cos(\psi-\phi/2))^2}}\right)}_{\sin^{-1}(a)}(1-a\sin(t))\cos(t)dtd\psi.
\end{align}
Therefore, 
\be\label{IntComp2}
\int_{R_b}x\cdot ndS=\int_{R_c}x\cdot ndS.
\ee


\section{Spindle integrals}\label{SpindleSect}
We now aim to evaluate the integral 
$$
\int_{R_s}x\cdot ndS. 
$$
Recall that $R_s$ is the surface obtained by rotating the geodesic $\gamma_b\subset \partial B(b)$ into the geodesic $\gamma_c\subset \partial B(c)$ counterclockwise about the line which passes through $b'$ and $c'$ with direction $b'-c'$. We will proceed analogously to the argument above by finding an appropriate parametrization.  We will initially do this in general and then specialize to the specific coordinate system \eqref{SpecialCoordinates}. 
\\
\par {\bf Geodesics.} First note that we may parametrize $\gamma_b$ as 
$$
\gamma_b(t)=b+\cos(t)(b'-b)+\sin(t)\left(\frac{c'-b-\cos(\theta')(b'-b)}{\sin(\theta')}\right)
$$
for $0\le t\le \theta'$. Observe that $\gamma_b(0)=b'$, $\gamma_b(\theta')=c'$, and $|\dot\gamma_b(t)|=1$ for $t\in [0,\theta']$. We may also parametrize $\gamma_c$ as 
$$
\gamma_c(t)=c+\cos(t)(b'-c)+\sin(t)\left(\frac{c'-c-\cos(\theta')(b'-c)}{\sin(\theta')}\right)
$$
for $0\le t\le \theta'$. This is also a constant speed geodesic which starts at $b'$ and ends at $c'$.  
\\
\par {\bf The edge $e$.} The edge $e$ is an arc of the circle $\partial B(b')\cap \partial B(c')$ given by the equations 
$$
\left|x-\frac{b'+c'}{2}\right|=\sqrt{1-\left|\frac{b'-c'}{2}\right|^2}=\cos(\theta'/2)
$$
and 
$$
\left(x-\frac{b'+c'}{2}\right)\cdot (b'-c')=0
$$
for $x\in \R^3$. Consequently, $e$   may be parametrized as

$$
\eta(s)=\frac{b'+c'}{2}+\cos(s)\left(b-\frac{b'+c'}{2}\right) +
\sin(s)\;v\times \left(b-\frac{b'+c'}{2}\right)
$$
for $0\le s\le \phi'$. Here we have set 
$$
v=\frac{b'-c'}{|b'-c'|}.
$$
Note $\eta(0)=b$, and \eqref{OrientedPoints} can be used to check that $\eta(\phi')=c$.

\par Observe that $\eta(s)-(b'+c')/2$ is the counterclockwise rotation through angle $s$ of $b-(b'+c')/2$ about 
the line which passes through $b'$ and $c'$ in the $v$ direction. This can also be seen from the identity
\be\label{etaDerivative}
\eta'(s)=v\times \left(\eta(s)-\frac{b'+c'}{2}\right),
\ee
which in turn shows $\eta$ has constant speed
$$
|\dot \eta(s)|=\cos(\theta'/2)
$$
for $0\le s\le \phi'$. 
\\
\par {\bf Parametrization.} Define the mapping
$$
X(s,t)=\eta(s)+\cos(t)(b'-\eta(s))+\sin(t)\left(\frac{c'-\eta(s)-\cos(\theta')(b'-\eta(s))}{\sin(\theta')}\right)
$$
 for $(s,t)\in [0,\phi']\times[0,\theta']$. We claim this is a parametrization of $R_s$.  Clearly 
$$
X(0,t)=\gamma_b(t)\quad\text{and}\quad X(\phi',t)=\gamma_c(t)
$$
for $t\in [0,\theta']$. 
\par Next we observe that  
$$
b'-\eta(s)=\frac{b'-c'}{2}-\left(\eta(s)-\frac{b'+c'}{2}\right)=\sin(\theta'/2)v-\left(\eta(s)-\frac{b'+c'}{2}\right)
$$
and
$$
c'-\eta(s)=-\frac{b'-c'}{2}-\left(\eta(s)-\frac{b'+c'}{2}\right)=-\sin(\theta'/2)v-\left(\eta(s)-\frac{b'+c'}{2}\right)
$$
for $0\le s\le \phi'$.  A routine manipulation then gives the identity 
\be\label{XparamIdentity}
 X(s,t)-\frac{b'+c'}{2}=-\sin(t-\theta'/2)v-\frac{\cos(t-\theta'/2)-\cos(\theta'/2)}{\cos(\theta'/2)}\left(\eta(s)-\frac{b'+c'}{2}\right).
\ee
\par It follows that $ X(s,t)-(b'+c')/2$ is the counterclockwise rotation of $\gamma_b(t)-(b'+c')/2$ through angle $s$ about 
the line which passes through $b'$ and $c'$ in the $v$ direction. As a result, the mapping $X$ parametrizes $R_s$. It is also possible to check directly that $X(s,t)$ satisfies equation \eqref{spindleSurface} and inequalities \eqref{spindleCondition} 
 for $(s,t)\in [0,\phi']\times[0,\theta']$.
\\
\par {\bf Surface area element.}  In view of \eqref{XparamIdentity},  
\be
\partial_sX(s,t)=-\frac{\cos(t-\theta'/2)-\cos(\theta'/2)}{\cos(\theta'/2)}\eta'(s)
\ee
and 
\be
 \partial_tX(s,t)=-\cos(t-\theta'/2)v+\frac{\sin(t-\theta'/2)}{\cos(\theta'/2)}\left(\eta(s)-\frac{b'+c'}{2}\right).
\ee
By \eqref{etaDerivative}, 
$$
\partial_sX(s,t)\cdot \partial_tX(s,t)=0.
$$ 
Therefore, 
$$
|\partial_sX(s,t)\times \partial_tX(s,t)|=\cos(t-\theta'/2)-\cos(\theta'/2).
$$
In particular, we note 
\begin{align}
S(R_s)=\int^{\phi'}_0\int^{\theta'}_0|\partial_sX(s,t)\times \partial_tX(s,t)|dtds=2\phi'\Big(\sin(\theta'/2)-\cos(\theta'/2)\;\theta'/2\Big).
\end{align}
\\
\par {\bf Unit normal.}  Recall that $\eta(s)-(b'+c')/2$ is orthogonal to $v$ and the length of $\eta(s)-(b'+c')/2$ is $\cos(\theta'/2)$ for each $s\in [0,\phi']$.  It follows that 
$$
\eta'(s)\times v=\left[v\times \left(\eta(s)-\frac{b'+c'}{2}\right)\right] \times v=\eta(s)-\frac{b'+c'}{2}
$$
and 
$$
\eta'(s)\times\left(\eta(s)-\frac{b'+c'}{2}\right)= \left[v\times \left(\eta(s)-\frac{b'+c'}{2}\right)\right] \times \left(\eta(s)-\frac{b'+c'}{2}\right)=-\cos(\theta'/2)^2v.
$$
As $[0,\theta']\ni t\mapsto X(s,t)$ traces a geodesic in $\partial B(\eta(s))$ from $b'$ to $c'$ and $[0,\phi']\ni s\mapsto X(s,t)$ rotates 
$\gamma_b(t)$ counterclockwise around the line passing through $b'$ and $c'$ in the direction $v$, 
\begin{align}
n\circ X(s,t)&=\frac{\partial_sX(s,t)\times \partial_tX(s,t)}{|\partial_sX(s,t)\times \partial_tX(s,t)|}\\
&=-\frac{\eta'(s)}{\cos(\theta'/2)}\times \left(-\cos(t-\theta'/2)v+\frac{\sin(t-\theta'/2)}{\cos(\theta'/2)}\left(\eta(s)-\frac{b'+c'}{2}\right)\right)\\
&=\frac{\cos(t-\theta'/2)}{\cos(\theta'/2)}\left(\eta(s)-\frac{b'+c'}{2}\right)+\sin(t-\theta'/2)v.
\end{align}
\begin{rem}
This is the outward unit normal to $R_s\subset \partial W(e')$. Note in particular that this is the inward unit normal to the spindle surface described 
by equation \eqref{spindleSurface}.  
\end{rem}

\subsection{Volume integral} 
We will write 
$$
\int_{R_s}x\cdot n dS=\int_{R_s}\left(x-\frac{b'+c'}{2}\right)\cdot n dS+\left(\frac{b'+c'}{2}\right)\cdot\int_{R_s} n dS
$$
and evaluate both integrals separately. First note that 
\begin{align}
&\int_{R_s}\left(x-\frac{b'+c'}{2}\right)\cdot n dS\\
&=\int^{\phi'}_0\int^{\theta'}_0\left(X(s,t)-\frac{b'+c'}{2}\right)\cdot n\circ X(s,t)|\partial_sX(s,t)\times \partial_tX(s,t)|dtds \\
&=\phi'\int^{\theta'}_0\left(-1+\cos(\theta'/2)\cos(t-\theta'/2)\right)(\cos(t-\theta'/2)-\cos(\theta'/2))dt \\
&=-\phi'\Big(3\sin(\theta'/2)-3\cos(\theta'/2)\;\theta'/2 -\sin(\theta'/2)^3\Big).
\end{align}

\par Next observe 
\begin{align}
&\int_{R_s}n dS\\
&= \int^{\phi'}_0\int^{\theta'}_0n\circ X(s,t)|\partial_sX(s,t)\times \partial_tX(s,t)|dtds\\
&= \int^{\phi'}_0\int^{\theta'}_0\left(\frac{\cos(t-\theta'/2)}{\cos(\theta'/2)}\left(\eta(s)-\frac{b'+c'}{2}\right)+\sin(t-\theta'/2)v\right) (\cos(t-\theta'/2)-\cos(\theta'/2))dtds\\
&= \int^{\phi'}_0\int^{\theta'/2}_{-\theta'/2}\left(\frac{\cos(\tau)}{\cos(\theta'/2)}\left(\eta(s)-\frac{b'+c'}{2}\right)+\sin(\tau)v\right) (\cos(\tau)-\cos(\theta'/2))d\tau ds\\
&= \int^{\phi'}_0\int^{\theta'/2}_{-\theta'/2}\left(\frac{\cos(\tau)}{\cos(\theta'/2)}\left(\eta(s)-\frac{b'+c'}{2}\right)\right) (\cos(\tau)-\cos(\theta'/2))d\tau ds\\
&=2\int^{\theta'/2}_{0}\left(\frac{\cos(\tau)^2}{\cos(\theta'/2)}-\cos(\tau)\right)d\tau \int^{\phi'}_0\left(\eta(s)-\frac{b'+c'}{2}\right)  ds\\
&=\left(\frac{\theta'/2}{\cos(\theta'/2)}-\sin(\theta'/2)\right)\int^{\phi'}_0\left(\eta(s)-\frac{b'+c'}{2}\right)  ds.
\end{align}
We also have 
\begin{align}
&\int^{\phi'}_0\left(\eta(s)-\frac{b'+c'}{2}\right) ds\\
&=\int^{\phi'}_0\cos(s)ds\; \left(b-\frac{b'+c'}{2}\right)
+\int^{\phi'}_0\sin(s)ds\; v\times \left(b-\frac{b'+c'}{2}\right)\\
&=\sin(\phi')\; \left(b-\frac{b'+c'}{2}\right)
+(1-\cos(\phi'))\; v\times \left(b-\frac{b'+c'}{2}\right)\\
&=2\sin(\phi'/2)\left[\cos(\phi'/2) \left(b-\frac{b'+c'}{2}\right)
+\sin(\phi'/2) \; v\times \left(b-\frac{b'+c'}{2}\right)\right].
\end{align}
That is, 
\begin{align}
&\left(\frac{b'+c'}{2}\right)\cdot\int_{R_s} n dS\\
&\hspace{.5in}= \left(\frac{\theta'/2}{\cos(\theta'/2)}-\sin(\theta'/2)\right)2\sin(\phi'/2)\left[\cos(\phi'/2)  \left(\frac{b'+c'}{2}\right)\cdot\left(b-\frac{b'+c'}{2}\right) \right.\\
&\hspace{3.25in}\left.+\sin(\phi'/2) \left(\frac{b'+c'}{2}\right)\cdot v\times \left(b-\frac{b'+c'}{2}\right)\right].
\end{align}

\subsection{Special coordinates} 
We will make the expressions
$$
\left(\frac{b'+c'}{2}\right)\cdot\left(b-\frac{b'+c'}{2}\right) \quad \text{and}\quad \left(\frac{b'+c'}{2}\right)\cdot v\times \left(b-\frac{b'+c'}{2}\right)
$$
more explicit using the specific coordinates \eqref{SpecialCoordinates}. First note
\begin{align}
\frac{b'+c'}{2}&=\sqrt{1-a^2}\left(\frac{1+\cos(\phi)}{2},\frac{\sin(\phi)}{2},0\right)=\sqrt{1-a^2}\cos(\phi/2)\left(\cos(\phi/2),\sin(\phi/2),0\right).
\end{align}
It follows that 
$$
b-\frac{b'+c'}{2}=\left(-\sqrt{1-a^2}\cos(\phi/2)\cos(\phi/2),-\sqrt{1-a^2}\cos(\phi/2)\sin(\phi/2),a\right),
$$
and 
\begin{align}
\left(\frac{b'+c'}{2}\right)\cdot\left(b-\frac{b'+c'}{2}\right)&=-(1-a^2)\cos(\phi/2)^2=-\cos(\theta/2)^2\cos(\phi/2)^2.
\end{align}

\par We also have 
\begin{align}
\frac{b'-c'}{2}&=\sqrt{1-a^2}\left(\frac{1-\cos(\phi)}{2},-\frac{\sin(\phi)}{2},0\right)=\sqrt{1-a^2}\sin(\phi/2)\left(\sin(\phi/2),-\cos(\phi/2),0\right).
\end{align}
This implies 
$$
v=\frac{b'-c'}{|b'-c'|}=\left(\sin(\phi/2),-\cos(\phi/2),0\right).
$$
In addition, a direct computation gives 
$$
v\times \left(b-\frac{b'+c'}{2}\right)=
\left(-a\cos(\phi/2) , -a\sin(\phi/2) , -\sqrt{1-a^2}\cos(\phi/2) \right).
$$
As a consequence, 
\begin{align}
&\left(\frac{b'+c'}{2}\right)\cdot v\times \left(b-\frac{b'+c'}{2}\right) =-a\sqrt{1-a^2}\cos(\phi/2) =-\sin(\theta/2)\cos(\theta/2)\cos(\phi/2).
\end{align}

\par Combining these computations leads to 
 \begin{align}
&2\sin(\phi'/2)\left[\cos(\phi'/2)  \left(\frac{b'+c'}{2}\right)\cdot\left(b-\frac{b'+c'}{2}\right) +\sin(\phi'/2) \left(\frac{b'+c'}{2}\right)\cdot v\times \left(b-\frac{b'+c'}{2}\right)\right]\\
&=-2\sin(\phi'/2) \cos(\theta/2)\cos(\phi/2)\Big(\cos(\phi'/2)\cos(\theta/2)\cos(\phi/2)+\sin(\phi'/2)\sin(\theta/2)\Big)\\
& =-2\frac{\sin(\theta/2)}{\cos(\theta'/2)} \cos(\theta/2)\cos(\phi/2)\Big(\cos(\phi'/2)\cos(\theta/2)\cos(\phi/2)+\frac{\sin(\theta/2)}{\cos(\theta'/2)}\sin(\theta/2)\Big)\\
& =-2\frac{\sin(\theta/2)}{\cos(\theta'/2)^2} \cos(\theta/2)\cos(\phi/2)\Big(\cos(\theta/2)^2\cos(\phi/2)^2+\sin(\theta/2)^2\Big)\\
& =-2\frac{\sin(\theta/2)}{\cos(\theta'/2)^2} \cos(\theta/2)\cos(\phi/2)\Big(\cos(\theta/2)^2(1-\sin(\phi/2)^2)+\sin(\theta/2)^2\Big)\\
& =-2\frac{\sin(\theta/2)}{\cos(\theta'/2)^2} \cos(\theta/2)\cos(\phi/2)\Big(1-\cos(\theta/2)^2\sin(\phi/2)^2\Big)\\
& =-2\frac{\sin(\theta/2)}{\cos(\theta'/2)^2} \cos(\theta/2)\cos(\phi/2)\cos(\theta'/2)^2\\
& =-2\sin(\theta/2) \cos(\theta/2)\cos(\phi/2).
\end{align}
Note that we used the identity \eqref{MidPointIdentity} in the third equality above. It follows that 
\begin{align}
\left(\frac{b'+c'}{2}\right)\cdot\int_{R_s}n dS&=2\left(\sin(\theta'/2)-\frac{\theta'/2}{\cos(\theta'/2)}\right)\sin(\theta/2) \cos(\theta/2)\cos(\phi/2)\\
&=2\sin(\theta'/2)\sin(\theta/2) \cos(\theta/2)\cos(\phi/2)-\theta'\sin(\theta/2)\frac{\cos(\theta/2)\cos(\phi/2)}{\cos(\theta'/2)}\\
&=2\sin(\theta'/2)\sin(\theta/2) \cos(\theta/2)\cos(\phi/2)-\theta'\sin(\theta/2)\cos(\phi'/2).
\end{align}

\par In summary,  
\begin{align}\label{IntComp3}
\int_{R_s}x\cdot ndS&=-\phi'\Big(3\sin(\theta'/2)-3\cos(\theta'/2)\;\theta'/2 -\sin(\theta'/2)^3\Big)\\
&\quad\quad +2\sin(\theta'/2)\sin(\theta/2) \cos(\theta/2)\cos(\phi/2)-\theta'\sin(\theta/2)\cos(\phi'/2)
\end{align}
in the coordinate system \eqref{SpecialCoordinates}.

\section{Proof of the Wedge Lemma}\label{ProofSect}
This section is dedicated to the proof of the Wedge Lemma; as previously noted, our volume formula \eqref{VolumeFormula} follows directly from 
\eqref{CrucialhIdentity}.  In view of \eqref{IntComp1}, \eqref{IntComp2}, and \eqref{IntComp3},
\begin{align}
&V(W(e'))\\
&=\frac{1}{3}\int_{R_b}x\cdot ndS+\frac{1}{3}\int_{R_c}x\cdot ndS+\frac{1}{3}\int_{R_s}x\cdot ndS\\
&=\frac{2}{3}\int_{R_c}x\cdot ndS+\frac{1}{3}\int_{R_s}x\cdot ndS\\
&=\frac{2}{3}\left(2\sin^{-1}\Big(\tan(\theta/2)\tan(\theta'/2)\Big)-\frac{3}{2}\phi\sin(\theta/2)+\sin(\theta/2)^3\frac{\phi}{2}+\sin(\theta/2)\cos(\phi'/2)\frac{\theta'}{2}\right)\\
&\quad-\phi'\left(\sin(\theta'/2)-\cos(\theta'/2)\;\theta'/2 -\frac{1}{3}\sin(\theta'/2)^3\right)\\
&\quad\quad +\frac{2}{3}\sin(\theta'/2)\sin(\theta/2) \cos(\theta/2)\cos(\phi/2)-\frac{1}{3}\theta'\sin(\theta/2)\cos(\phi'/2)\\
&=\frac{1}{2}f(\theta,\theta')+\frac{4}{3}\sin^{-1}\Big(\tan(\theta/2)\tan(\theta'/2)\Big)-\phi\left(\sin(\theta/2)-\frac{1}{3}\sin(\theta/2)^3\right)\\
&\quad\quad-\phi'\left(\sin(\theta'/2)-\frac{1}{3}\sin(\theta'/2)^3\right) +\frac{2}{3}\sin(\theta'/2)\sin(\theta/2) \cos(\theta/2)\cos(\phi/2).
\end{align}
We used $f(\theta,\theta')=\phi'\theta'\cos(\theta'/2)$ in the last equality. 

\par Recall 
$$
\sin(\phi/2)=\frac{\sin(\theta'/2)}{\cos(\theta/2)}\quad \text{and}\quad \sin(\phi'/2)=\frac{\sin(\theta/2)}{\cos(\theta'/2)}.
$$
It follows that 
\be\label{coscosIdentityAgain}
 \cos(\theta/2)\cos(\phi/2)= \cos(\theta/2)\sqrt{1-\sin(\phi/2)^2}=\sqrt{1-\sin(\theta/2)^2-\sin(\theta'/2)^2}
\ee
and 
\begin{align}
&V(W(e'))\\
&=\frac{1}{2}f(\theta,\theta')+\frac{4}{3}\sin^{-1}\Big(\tan(\theta/2)\tan(\theta'/2)\Big)-\phi\left(\sin(\theta/2)-\frac{1}{3}\sin(\theta/2)^3\right)\\
&\quad\quad-\phi'\left(\sin(\theta'/2)-\frac{1}{3}\sin(\theta'/2)^3\right) +\frac{2}{3}\sin(\theta'/2)\sin(\theta/2)\sqrt{1-\sin(\theta/2)^2-\sin(\theta'/2)^2}\\
 &=\frac{1}{2}f(\theta,\theta')-\frac{1}{2}h(\theta,\theta').
\end{align}

\appendix
\section{Volume and surface area inequality}
We will show that $h(\theta,\theta')> g(\theta,\theta')$ and derive the inequality \eqref{AlmostBlaschkeIneq} between the volume and surface area of a Reuleaux polyhedron. 
\begin{lem}
For each $\theta,\theta'\in (0,\pi/3]$, $h(\theta,\theta')> g(\theta,\theta')$.
\end{lem}
\begin{proof}
Set $a=\sin(\theta/2),$ $b=\sin(\theta'/2),$ and $d=\sqrt{1-a^2-b^2}$.  Then $0<a,b\leq 1/2$ and $d>0$. Employing the identity
\[
 \sin^{-1}x=\tan^{-1}\left(\frac{x}{\sqrt{1-x^2}}\right)
 \qquad (0\leq x<1)
\]
in the definitions of $h$ and $g$ we obtain
\[
h(\theta,\theta')-g(\theta,\theta')
=\frac{4}{3}F(a,b),
\]
where
\[
F(a,b)=
\tan^{-1}\left(\frac{ab}{d}\right)
-a^3\tan^{-1}\left(\frac{b}{d}\right)
-b^3\tan^{-1}\left(\frac{a}{d}\right)-abd.
\]
Note $F(0,b)=0$. Direct computation also gives 
\[
\frac{\partial F}{\partial a}(a,b)
=3a^2\left(\frac{b}{d}-\tan^{-1}\left(\frac{b}{d}\right)\right)>0,
\]
since $\tan^{-1} x<x$ for $x>0$.  Hence $F(a,b)>0$, which proves the
claim.

\end{proof}
\begin{prop}
Suppose $X\subset \R^3$ is extremal. Then 
$$
V(B(X))<\frac{1}{2}S(B(X))-\frac{\pi}{3}.
$$
\end{prop}
\begin{proof}
Suppose  $|X|=m$ and $(e_1,e_1'),\dots, (e_{m-1},e_{m-1}')$  are the dual edge pairs of $B(X)$. 
Then 
\begin{align}
V(B(X))&=\frac{2\pi}{3}-\frac{1}{2}\sum^{m-1}_{i=1}h(\theta(e_i),\theta(e_i'))\\
&< \frac{2\pi}{3}-\frac{1}{2}\sum^{m-1}_{i=1}g(\theta(e_i),\theta(e_i'))\\
&=\frac{1}{2}\left(2\pi-\sum^{m-1}_{i=1}g(\theta(e_i),\theta(e_i')) \right)-\frac{\pi}{3}\\
&=\frac{1}{2}S(B(X))-\frac{\pi}{3}.
\end{align}
\end{proof}


\bibliography{volbib}{}

\bibliographystyle{plainurl}

\typeout{get arXiv to do 4 passes: Label(s) may have changed. Rerun}

\end{document}